\documentclass[12pt,a4paper]{amsart}
\usepackage[bottom]{footmisc}
\usepackage{mathrsfs}
\usepackage{amsmath}
\usepackage{amsthm}
\usepackage{amsfonts}
\usepackage{latexsym}
\usepackage{graphicx}
\usepackage{hyperref}
\usepackage{amssymb}
\usepackage{subfigure}
\usepackage{epsf}
\usepackage{float}
\usepackage{hyperref}
\usepackage{fancyhdr}
\allowdisplaybreaks \makeatletter
\def\rightharpoonfill@{\arrowfill@\relbar\relbar\rightharpoonup}
\DeclareRobustCommand{\overrightharpoon}{\mathpalette{\underarrow@\rightharpoonfill@}}
\makeatother

\allowdisplaybreaks

\begin{document}
\newcommand{\beq}{\begin{equation}}
\newcommand{\eneq}{\end{equation}}
\newtheorem{thm}{Theorem}[section]
\newtheorem{coro}[thm]{Corollary}
\newtheorem{lem}[thm]{Lemma}
\newtheorem{prop}[thm]{Proposition}
\newtheorem{defi}[thm]{Definition}
\newtheorem{rem}[thm]{Remark}
\newtheorem{cl}[thm]{Claim}
\title{Analytic solutions for the approximated 1-D Kantorovich mass transfer problems}
\author{Xiaojun Lu$^{1}$\ \ \ \ Xiaofen Lv$^2$}
\pagestyle{fancy}                   
\lhead{X. Lu and X. Lv}
\rhead{Monge-Kantorovich mass transfer problem} 
\thanks{Corresponding author: Xiaojun Lu, Department of Mathematics \& Jiangsu Key Laboratory of
Engineering Mechanics, Southeast University, 210096, Nanjing, China}
\thanks{Email addresses:  lvxiaojun1119@hotmail.de(Xiaojun Lu), lvxiaofen0101@hotmail.com(Xiaofen Lv)}
\date{}
\maketitle
\begin{center}
1. Department of Mathematics \& Jiangsu Key Laboratory of
Engineering Mechanics, Southeast University, 210096, Nanjing, China\\
2. Jiangsu Testing Center for Quality of Construction Engineering
Co., Ltd, 210028, Nanjing, China
\end{center}
\begin{abstract}
This paper mainly investigates the approximation of a global
maximizer of the 1-D Monge-Kantorovich mass transfer problem through
the approach of nonlinear differential equations with Dirichlet
boundary. Using an approximation mechanism, the primal maximization
problem can be transformed into a sequence of minimization problems.
By applying the canonical duality theory, one is able to derive a
sequence of analytic solutions for the minimization problems. In the
final analysis, the convergence of the sequence to a global
maximizer of the primal Monge-Kantorovich problem will be
demonstrated.
\end{abstract}
\section{Introduction}
Mass transfer is the net movement of mass from one location to
another by the action of driving forces, such as pressure
gradient(pressure diffusion), temperature gradient(thermal
diffusion), etc. In our physical world, when a system contains more
components with various concentration from point to point, a natural
tendency for mass to be transferred occurred in order to minimize
any concentration difference within the system. This transfer
phenomenon is governed by Fick's First Law. The original transfer
problem, which was proposed by Monge \cite{Monge}, investigated how
to move one mass distribution to
another one with the least amount of work by searching for a mapping ${\bf s}$ to
minimize the cost functional $$C[{\bf r}]:=\int_{\Omega}|x-{\bf r}(x)|d\mu^+(x)$$\\
among the 1-1 mappings ${\bf r}:\Omega\to\Omega^*$ that push forward
$\mu^+$ into $\mu^-$, where both $\Omega$ and $\Omega^*$ are bounded
domains in $\mathbb{R}^n$, $\mu^+$ and $\mu^-$ are two nonnegative
Radon measures on $\Omega$ and $\Omega^*$, respectively.\\

In the 1940s, Kantorovich initiated a duality theory by relaxing
Monge's transfer problem to the task of finding a global
maximizer(so-called {\it Kantorovich potential}) for the Kantorovich
problem in the following form \cite{K1,K2}, \beq(\mathscr{P}):
\displaystyle\max_{u}\Big\{K[u]:=\int_{U}ufdx\Big\},\eneq where
$U=\Omega\cup\Omega^*$, $f:=f^+-f^-$, $f^+\in C(\overline{\Omega})$
and $f^-\in C(\overline{\Omega^*})$ are two nonnegative density
functions and satisfy the normalized balance condition
\[
\int_\Omega f^+dx=\int_{\Omega^*}f^-dx=1.
\]
$u$ is subject to the following constraints, \beq u\in
W_0^{1,\infty}(U)\cap C(\overline{U}), \eneq \beq u=0\ \text{on}\
\overline{\Omega\cap\Omega^*},\eneq \beq\|
u_x\|_{L^\infty(U)}\leq1.\eneq In particular, when
$\Omega\cap\Omega^*=\emptyset$, $C(\overline{U})$ represents
$C(\overline{\Omega})$ and $C(\overline{\Omega^*})$, respectively.\\

As a matter of fact, the Kantorovich problem may not be a perfect
dual to the Monge problem unless a so-called {\it dual criteria for
optimality} is satisfied \cite{Ca1,Evans1}. Nowadays, the
Monge-Kantorovich mass transfer model is widely used in the
diffusive and convective transport of chemical species, purification
of blood in the kidneys and livers, separation of chemical
components in distillation columns, controlling haze in the
atmosphere by artificial precipitation, etc. Interested readers can
refer to \cite{LA1,LA2,Evans1,K1,K2,Monge,Su} for more
important applications. \\

Indeed, many mathematical tools have been developed for the
infinite-dimensional linear programming \cite{R,Su,V}, etc. In this
paper, we consider the 1-D Monge-Kantorovich problem (1) through a
nonlinear differential equation approach by introducing a sequence
of approximation problems for the primal problem $(\mathscr{P})$,
\begin{equation}(\mathscr{P}^{(k)}):
\displaystyle\min_{w_k}\Big\{I^{(k)}[w_k]:=\int_U L^{(k)}(
w_{k,x},w_k,x)dx=\int_U \Big(H^{(k)}(w_{k,x})-w_kf\Big)dx\Big\},
\end{equation}
where $w_{k,x}$ is the weak derivative of $w_k$ with respect to $x$,
$H^{(k)}:\mathbb{R}\to\mathbb{R}^+$ is defined as
$$H^{(k)}(x):={\rm e}^{k(x^2-1)/2}/k.$$
In particular, $I^{(k)}$ is called the {\it potential energy
functional} and is weakly lower semicontinuous on
$W^{1,\infty}_0(U)$. Moreover,
$L^{(k)}(P,z,x):\mathbb{R}\times\mathbb{R}\times U\to \mathbb{R}$
satisfies the following coercivity inequality and is convex in the
variable $P$, $$ L^{(k)}(P,z,x)\geq p_{k}P^2-q_k,\ P\in\mathbb{R},
z\in\mathbb{R}, x\in U, $$ for certain constants $p_k$ and $q_k$.
Notice that when $|x|\leq1$, then
$\displaystyle\lim_{k\to\infty}H^{(k)}(x)=0$ uniformly. All these
facts assures the existence of the global minimizer for
$\mathscr{P}^{(k)}$ \cite{Evans4}. Once such a sequence of global
minimizers $\{\bar{u}_k\}_k$ is obtained, then it will help find a
global Kantorovich potential which solves the primal problem
$(\mathscr{P})$, as is explained in
Theorem 1.2.\\

The key mission of this paper is to obtain an explicit
representation of this approximation sequence $\{\bar{u}_k\}_k$. By
variational calculus, one derives a correspondingly sequence of
Euler-Lagrange equations for $(\mathscr{P}^{(k)})$, namely, for any
$k\in\mathbb{N}$, \beq
\begin{array}{ll}\displaystyle ({\rm e}^{k(u_{k,x}^2-1)/2}u_{k,x})_x+f=0,& \ \text{\rm in}\
U\setminus\{\overline{\Omega\cap\Omega^*}\},
\end{array}\eneq
equipped with the Dirichlet boundary condition. The term ${\rm
e}^{k(\nabla u_k^2-1)/2}$ is called the {\it transport density}.
Actually, (6) is a highly nonlinear differential equation which is
difficult to solve by the direct approach \cite{JH,Evans4,LIONS}.
However, by the canonical duality theory, one is able to demonstrate
the existence and uniqueness of the solution for the nonlinear
differential equation, which establishes the equivalence between the
global minimizer of ($\mathscr{P}^{(k)}$)
and the solution of Euler-Lagrange equation (6).\\

In the former literature, as far as we know, few authors considered
the analytic Kantorovich potential of the approximation problems
($\mathscr{P}^{(k)}$), either a-priori estimates or numerical
approach \cite{Evans1,Evans2,Evans3,R}. The purpose of this paper is
to investigate the analytic solutions for the optimization problems
($\mathscr{P}^{(k)}$) through {\it canonical duality theory}
introduced by G. Strang et al. \cite{G1}. These methods were
originally proposed to find global minimizers for a non-convex
strain energy functional with a double-well potential, which was a
very challenging nonconvex problem.\\

At the moment, we would like to introduce the main theorems.
\begin{thm}
For any positive density functions $f^+\in C(\overline{\Omega})$ and
$f^-\in C(\overline{\Omega^*})$ satisfying the normalized balance
condition, there exists a sequence of solutions $\{\bar{u}_k\}_k$
subject to the constraints (2)-(4) for the Euler-Lagrange equations
(6), which is at the same time a sequence of global minimizers for
the approximation problems ($\mathscr{P}^{(k)}$). In particular,
when $\Omega=(a,b)$, $\Omega^*=(c,d)$,
$\Omega\cap\Omega^*=\emptyset$, then $\{\bar{u}_k\}_k$ can be
represented explicitly as
\[
\bar{u}_k(x)= \left\{
\begin{array}{lll}
\displaystyle\int^{x}_{a}(-F^+(t)+C_k)/E_k^{-1}((-F^+(t)+C_k)^2)dt,&
x\in[a,b],\\
\\
\displaystyle\int^{x}_{c}(F^-(t)-D_k)/E_k^{-1}((F^-(t)-D_k)^2)dt,&
x\in[c,d],
\end{array}
\right.
\]
where $E_k$, $F$ and $G$ are defined as
\[
\left\{
\begin{array}{lll}
E_k(x):=\displaystyle x^2\ln({\rm e}x^{2/k}),&
x\in(0,1],\\
\\
F^+(x):=\displaystyle\int^{x}_{a}f^+(t)dt,&
x\in[a,b],\\
\\
F^-(x):=\displaystyle\int^{x}_{c}f^-(t)dt,& x\in[c,d],\\
\end{array}
\right.
\]
$E_k^{-1}$ stands for the inverse of $E_k$, both $\{C_k\}_k$ and
$\{D_k\}_k$ are number sequences in $(0,1)$.
\end{thm}
By Rellich-Kondrachov Compactness Theorem, we have the following
convergence result.
\begin{thm}
For any positive density functions $f^+\in C(\overline{\Omega})$ and
$f^-\in C(\overline{\Omega^*})$ satisfying the normalized balance
condition, there exists a global minimizer subject to (2)-(4) for
the Kantorovich problem ($\mathscr{P}$).
\end{thm}

The rest of the paper is organized as follows. In Section 2, first,
we introduce some useful notations which will simplify our proof
considerably. Then, we apply the canonical dual transformation to
deduce a sequence of perfect dual problems ($\mathcal{P}^{(k)}_d$)
corresponding to $(\mathcal{P}^{(k)})$ and a pure complementary
energy principle. Next, we apply the canonical duality theory to
prove Theorem 1.3. In the final analysis, a global maximizer of the
primal Monge-Kantorovich problem will be given by the approximation
techniques in the proof of Theorem 1.4.
\section{Proof of the main results}
\subsection{Some useful notations}

\begin{itemize}
\item $\theta_k$ is given by
\[
\theta_k(x):={\rm e}^{k(w_{k,x}^2-1)/2}w_{k,x}.
\]
\item $\Phi^{(k)}$ is a nonlinear
geometric mapping defined as
\[
\Phi^{(k)}(w_k):=k(w_{k,x}^2-1)/2.
\]
For convenience's sake, denote $\xi_k:=\Phi^{(k)}(w_k).$ It is
evident that $\xi_k$ belongs to the function space $\mathscr{U}$
given by
\[
\mathscr{U}:= \Big\{\phi\in L^\infty(U)\Big| \phi\leq 0\Big\}.
\]
\item $\Psi^{(k)}$ is a canonical energy
defined as
\[
\Psi^{(k)}(\xi_k):={\rm e}^{\xi_k}/k,
\]
which is a convex function with respect to $\xi_k$. For simplicity,
denote $\zeta_k:=\Psi^{(k)}(\xi_k)$, which is the G\^{a}teaux
derivative of $\Psi^{(k)}$ with respect to $\xi_k$. Moreover,
$\zeta_k$ is invertible with respect to $\xi_k$ and belongs to the
function space $\mathscr{V}^{(k)}$,
\[\mathscr{V}^{(k)}:=\Big\{\phi\in
L^\infty(U)\Big| 0<\phi\leq 1/k\Big\}.
\]
\item
$\Psi^{(k)}_\ast$ is defined as
\[
\Psi^{(k)}_\ast(\zeta_k):=\xi_k\zeta_k-\Psi^{(k)}(\xi_k)=\zeta_k(\ln(k\zeta_k)-1).
\]
\item $\lambda_k$ is defined as $\lambda_k:=k\zeta_k,$ and belongs
to the function space $\mathscr{V}$,
\[\mathscr{V}:=\Big\{\phi\in
L^\infty(U)\Big| 0<\phi\leq 1\Big\}.
\]
\end{itemize}
\subsection{Proof of Theorem 1.1}
Before we prove the main result, first and foremost, we give some
useful definitions.
\begin{defi}
By Legendre transformation, one defines a total complementary energy
functional $\Xi^{(k)}$,
\[
\Xi^{(k)}(u_k,\zeta_k):=\displaystyle\int_{U}\Big\{\Phi^{(k)}(u_k)\zeta_k-\Psi^{(k)}_\ast(\zeta_k)
-fu_k\Big\}dx.
\]
\end{defi}
Next we introduce an important {\it criticality criterium} for the
total complementary energy functional.
\begin{defi}
$(\bar{u}_k, \bar{\zeta}_k)$ is called a critical pair of
$\Xi^{(k)}$ if and only if \beq
D_{u_k}\Xi^{(k)}(\bar{u}_k,\bar{\zeta}_k)=0, \eneq and \beq
D_{\zeta_k}\Xi^{(k)}(\bar{u}_k,\bar{\zeta}_k)=0, \eneq where
$D_{u_k}, D_{\zeta_k}$ denote the partial G\^ateaux derivatives of
$\Xi^{(k)}$, respectively. \end{defi} In effect, by variational
calculus, we have the following observations from (7) and (8).
\begin{lem}
On the one hand, for any fixed $\zeta_k\in\mathscr{V}^{(k)}$, $(7)$
is equivalent to the equilibrium equation
\[
\begin{array}{ll}\displaystyle(k\zeta_k \bar{u}_{k,x})_x+f=0,& \
\text{\rm in}\
U\setminus\{\overline{\Omega\cap\Omega^*}\}.\end{array}
\]
On the other hand, for any fixed $u_k$ subject to (2)-(4), (8) is
consistent with the constructive law
\[
\Phi^{(k)}(u_k)=D_{\zeta_k}\Psi^{(k)}_\ast(\bar{\zeta}_k).
\]
\end{lem}
Lemma 2.3 indicates that $\bar{u}_k$ from the critical pair
$(\bar{u}_k,\bar{\zeta}_k)$ solves the Euler-Lagrange equation (7).
\begin{defi}
From Definition 2.1, one defines a pure complementary energy
$I^{(k)}_d$ in the form
\[
I^{(k)}_d[\zeta_k]:=\Xi^{(k)}(\bar{u}_k,\zeta_k),
\]
where $\bar{u}_k$ solves the Euler-Lagrange equation (6).
\end{defi}
For convenience's sake, we show another representation of the pure
energy $I^{(k)}_d$.
\begin{lem} The
pure complementary energy functional $I^{(k)}_d$ can be rewritten as
\[
I^{(k)}_d[\zeta_k]=-1/2\int_{U}\Big\{{|\theta_k|^2/(k\zeta_k)}+k\zeta_k+2\zeta_k(\ln(k\zeta_k)-1)\Big\}dx,
\]
where $\theta_k$ satisfies \beq \theta_{k,x}+f=0\ \ \text{\rm in}\
U\setminus\{\overline{\Omega\cap\Omega^*}\}, \eneq equipped with a
hidden boundary condition.
\end{lem}
\begin{proof}
Through integrating by parts, one has
\[
\begin{array}{lll}
I^{(k)}_d[\zeta_k]&=&\displaystyle-\underbrace{\int_U\Big\{(k\zeta_k\bar{u}_{k,x})_x+f\Big\}\bar{u}_kdx}_{(I)}\\
\\
&&-\underbrace{1/2\int_U\Big\{k\zeta_k\bar{u}_{k,x}^2+k\zeta_k+2\zeta_k(\ln(k\zeta_k)-1)\Big\}dx.}_{(II)}\\
\\
\end{array}
\]
Since $\bar{u}_k$ solves the Euler-Lagrange equation (6), then, the
first part $(I)$ disappears. Keeping in mind the definition of
$\theta_k$ and $\zeta_k$, one reaches the conclusion.
\end{proof}

With the above discussion, next we establish a sequence of dual
variational problems corresponding to the approximation problems
($\mathscr{P}^{(k)}$).
\begin{equation}
(\mathscr{P}_d^{(k)}):\displaystyle\max_{\zeta_k\in\mathscr{V}^{(k)}}\Big\{I^{(k)}_d[\zeta_k]=-1/2\int_{U}\{{\theta_k^2/(k\zeta_k)}+k\zeta_k+2\zeta_k(\ln(k\zeta_k)-1)\}dx\Big\}.
\end{equation}
In effect, by calculating the G\^{a}teaux derivative of $I_d^{(k)}$
with respect to $\zeta_k$, we have \begin{lem} The variation of
$I_d^{(k)}$ with respect to $\zeta_k$ leads to the Dual Algebraic
Equation(DAE), namely, \beq
\theta_k^2=k{\bar{\zeta}_k}^2(2\ln(k\bar{\zeta}_k)+k), \eneq where
$\bar{\zeta}_k$ is from the critical pair
$(\bar{u}_k,\bar{\zeta}_k)$.
\end{lem}
As a matter of fact, the identity (11) can be rewritten as \beq
\theta_k^2=E_k(\lambda_k)={\lambda}_k^2\ln({\rm
e}{\lambda}_k^{2/k}). \eneq It is easy to check, $E_k$ is strictly
increasing with respect to $\lambda\in[{\rm e}^{-k/2},1]$.\\

From the above discussion, one deduces that, once $\theta_k$ is
given, then the analytic solution of the Euler-Lagrange equation (6)
can be represented as \beq
\bar{u}_k(x)=\displaystyle\int^{x}_{x_0}\eta_k(t)dt, \eneq where
$x\in \overline{U}, x_0\in\partial U$, $\eta_k=\theta_k/\lambda_k$.
Next, we verify that $\bar{u}_k$ is subject to (2)-(4) and is
exactly a global minimizer for ($\mathscr{P}^{(k)}$) and
$\bar{\zeta}_k$ is a global maximizer over $\mathscr{V}^{(k)}$ for
($\mathscr{P}^{(k)}_d$).
\begin{lem}(Canonical duality theory)
For any positive density functions $f^+\in C(\overline{\Omega})$ and
$f^-\in C(\overline{\Omega^*})$ satisfying the normalized balance
condition, there exists a unique sequence of solutions
$\{\bar{u}_k\}_k$ subject to (2)-(4) for the Euler-Lagrange
equations (6) with Dirichlet boundary in the form of (13), which is
a unique sequence of global minimizers for the approximation
problems ($\mathscr{P}^{(k)}$). And the corresponding
$\{\bar{\zeta}_k\}_k$ is a unique sequence of global maximizers for
the dual problems ($\mathscr{P}_d^{(k)}$). Moreover, the following
duality identity holds, \beq
I^{(k)}[\bar{u}_k]=\displaystyle\min_{u_k}I^{(k)}[u_k]=\Xi^{(k)}(\bar{u}_k,\bar{\zeta}_k)=\displaystyle\max_{\zeta_k}I_d^{(k)}[\zeta_k]=I_d^{(k)}[\bar{\zeta}_k].
\eneq
\end{lem}
\begin{rem}
Lemma 2.7 demonstrates that the maximization of the pure
complementary energy functional $I_d^{(k)}$ is perfectly dual to the
minimization of the potential energy functional $I^{(k)}$. In
effect, the identity (14) indicates there is no duality gap between
them.
\end{rem}
\begin{proof}
Without loss of generality, we consider the disjoint case
$\Omega=(a,b)$ and $\Omega^*=(c,d)$, $b<c$. We divide our proof into
three parts. In the first and second parts, we discuss the
uniqueness of $\theta_k$. Extremum conditions will be
illustrated in the third part.\\

{\it First Part:}\\

In $\Omega$, we have a general solution for the differential
equation (9) in the form of
\[
\theta_k(x)=-F^+(x)+C_k,\ \ x\in[a,b].
\]
Since $f^+>0$, then $F^+\in C[a,b]$ is a strictly increasing
function with respect to $x\in[a,b]$ and consequently is invertible.
From the identity (12), one sees that there exists a unique
continuous function $\lambda_k(x)\in[{\rm e}^{-k/2},1]$. By paying
attention to the Dirichlet boundary $\bar{u}_k(a)=0$, one has the
analytic solution $\bar{u}_k$ in the following form,
\[
\bar{u}_k(x)=\int^{x}_{a}\eta_k(x)dx,\ \ \ \ x\in[a,b].
\]
Since
\[
\displaystyle\lim_{x\to F^{-1}(C_k)}\eta_k(x)=0,
\]
thus, $\bar{u}_k\in C[a,b]$. Recall that
\[
\bar{u}_k(b)=\int^{F^{-1}(C_k)}_{a}\eta_k(x)dx+\int^{b}_{F^{-1}(C_k)}\eta_k(x)dx=0,
\]
and we can determine the constant $C_k\in(0,1)$ uniquely. Indeed,
let
\[
\mu_k(x,t):=(-F^+(x)+t)/\lambda_k(x,t)
\]
and
\[
M_{k}(t):=\int^{b}_{a}\mu_k(x,t)dx,
\]
where $\lambda_k(x,t)$ is from (12). It is evident that $\lambda_k$
depends on $C_k$. As a matter of fact, $M_k$ is strictly increasing
with respect to $t\in(0,1)$, which leads to
\[
C_k=M_{k}^{-1}(0).
\]

{\it Second Part:}\\

Applying the similar procedure, one sees that
\[
\theta_k(x)=F^-(x)-D_k, \ \ x\in[c,d],
\]
where the constant $D_k\in(0,1)$. Since $f^->0$, then $F^-\in
C[c,d]$ is a strictly increasing function with respect to
$x\in[c,d]$ and consequently is invertible. We can represent the
analytical solution $\bar{u}_k$ in the following form,
\[
\bar{u}_k(x)=\int^{x}_{c}\eta_k(x)dx,\ \ \ \ x\in[c,d],
\]
Since
\[
\displaystyle\lim_{x\to G^{-1}(D_k)}\eta_k(x)=0,
\]
thus, $\bar{u}_k\in C[c,d]$. Recall that
$$
\bar{u}_k(d)=0,$$ and we can determine the constant $D_k\in(0,1)$
uniquely. Indeed, let
\[
\rho_k(x,t):=(F^-(x)-t)/\lambda_k(x,t)
\]
and let
\[
N_k(t):=\int^{d}_{c}\rho_k(x,t)dx,
\]
where $\lambda_k(x,t)$ is from (12). As a matter of fact, $N_k$ is
strictly decreasing with respect to $t\in(0,1)$, which leads to
\[
D_k=N_{k}^{-1}(0).
\]
Furthermore, the other cases, such as $b=c$ and $b>c$ can also be
discussed similarly due to the fact that $\bar{u}_k=0$ on
$\overline{\Omega\cap\Omega^*}$. Therefore, $\theta_k$ is uniquely
determined
in $U$ and the analytic solution $\bar{u}_k\in C(\overline{U})$.\\

{\it Third Part:}\\

On the one hand, for any test function $\phi\in W^{1,\infty}_0$, the
second variational form $\delta_\phi^2I^{(k)}$ with respect to
$\phi$ is equal to\beq \int_U{\rm
e}^{k(\bar{u}_{k,x}^2-1)/2}\Big\{k(\bar{u}_{k,x}
\phi_x)^2+\phi_x^2\Big\}dx.\eneq On the other hand, for any test
function $\psi\in\mathscr{V}^{(k)}$, the second variational form
$\delta_\psi^2I_d^{(k)}$ with respect to $\psi$ is equal to
\beq-\int_U\Big\{\theta_k^2\psi^2/(k\zeta_k^3)+\psi^2/\zeta_k\Big\}dx.
\eneq From (15) and (16), one deduces immediately that
\[
\delta^2_\phi I^{(k)}(\bar{u}_k)\geq0,\ \
\delta_\psi^2J_d^{(k)}(\bar{\zeta}_k)\leq0.
\]
Together with the uniqueness of $\theta_k$ discussed in the first
and second parts, the proof is concluded.
\end{proof}
Consequently, we reach the conclusion of Theorem 1.1 by summarizing
the above discussion.
\subsection{Proof of Theorem 1.2}
Now we consider the convergence of the sequence $\{\bar{u}_k\}_k$ in
Theorem 1.1. According to Rellich-Kondrachov Compactness Theorem,
since
$$\displaystyle\sup_{k}|\bar{u}_k|\leq \text{\rm diam}(U)$$ and
$$\displaystyle\sup_{k}|\bar{u}_{k,x}|\leq 1,$$ then, there exists a
subsequence(without any confusion, we still denote as)
$\{\bar{u}_{k}\}_{k}$ and $u\in W_0^{1,\infty}(U)\cap
C(\overline{U})$ such that \beq\bar{u}_{k}\rightarrow u\
(k\to\infty)\ \text{in}\ L^\infty(U),\eneq \beq\bar{u}_{k,x}\
\overrightharpoon{*}\ u_{x}\ (k\to\infty)\ \text{weakly\ $\ast$\
in}\ L^\infty(U).\eneq From (18), one has
\[\|u_x\|_{L^\infty(U)}\leq\displaystyle\liminf_{k\to\infty}\|\bar{u}_{k,x}\|_{L^\infty(U)}\leq\sup_{k\to\infty}\|\bar{u}_{k,x}\|_{L^\infty(U)}\leq 1.\]
Consequently, one reaches the conclusion of Theorem 1.2 by
summarizing the above discussion.\\

{\bf Acknowledgment}: This project is partially supported by US Air
Force Office of Scientific Research (AFOSR FA9550-10-1-0487),
Natural Science Foundation of Jiangsu Province (BK 20130598),
National Natural Science Foundation of China (NSFC 71273048,
71473036, 11471072), the Scientific Research Foundation for the
Returned Overseas Chinese Scholars, Fundamental Research Funds for
the Central Universities on the Field Research of Commercialization
of Marriage between China and Vietnam (No. 2014B15214). This work is
also supported by Open Research Fund Program of Jiangsu Key
Laboratory of Engineering Mechanics, Southeast University
(LEM16B06). In particular, the authors also express their deep
gratitude to the referees for their careful reading and useful
remarks.

\end{document}